\documentclass[12pt,reqno]{amsart}
\usepackage{amsmath, amsfonts, amsthm, mathtools, fullpage, hyperref}
\usepackage[capitalise]{cleveref}
\usepackage[msc-links, alphabetic]{amsrefs}

\newcommand*{\map}[3]{#1 \colon #2 \to #3} 

\newcommand*{\abs}[1]{\left\vert #1 \right\vert} 

\newcommand*{\ip}[2]{\left\langle #1 , #2 \right\rangle} 

\newcommand*{\norm}[2][]{\left\Vert #2 \right\Vert_{#1}} 

\newcommand{\vol}{\ensuremath{\mathrm{vol}}}

\newtheorem{thm}{Theorem}[section]
\newtheorem{defn}[thm]{Definition}

\newtheorem{lemma}[thm]{Lemma}
\theoremstyle{remark}
\newtheorem*{remark}{Remark}

\makeatletter
\renewcommand\subsection{\@startsection{subsection}{2}%
  \z@{-.5\linespacing\@plus-.7\linespacing}{.5\linespacing}%
  {\normalfont\bfseries}}
\makeatother

\title{The dual Cheeger--Buser Inequality for Graphons}
\author{Mugdha Mahesh Pokharanakar}
\address{Department of Mathematics, Indian Institute of Science Education and Research Bhopal, 
Bhopal Bypass Road, Bhauri, Bhopal 462066, Madhya Pradesh, India}
\curraddr{}
\email{mugdha22@iiserb.ac.in}
\thanks{}
\date{\today}


\begin{document}

\begin{abstract}
    We introduce the notion of bipartiteness ratio for graphons. 
    We prove the dual Cheeger--Buser inequality for graphons, which relates the gap 
    between $2$ and the \emph{top of the spectrum} of the Laplacian of a graphon 
    with its bipartiteness ratio. The dual Cheeger--Buser inequality was established 
    by Trevisan and Bauer--Jost for graphs. Our result is an analog of that for graphons.
\end{abstract}

\maketitle

\section{Introduction}

A discrete version of the Cheeger--Buser inequality on Riemmanian manifolds 
was established for simple graphs by Dodziuk \cite{Dodziuk84}, 
Alon--Milman \cite{Alon-Milman85}, and Alon \cite{Alon-vertexCheeger86}. 
It gives a relation between the second smallest eigenvalue of the normalized Laplacian 
of a (finite undirected) graph and its Cheeger constant.

It is known that the largest eigenvalue of the Laplacian of a graph
is $2$ if and only if the graph has a bipartite connected component.
Trevisan \cite{Trevisan-MaxCut12} and Bauer--Jost \cite{Bauer-Jost13} defined 
the constants which measure ``bipartiteness'' of a graph, namely, the bipartiteness ratio 
$\beta(G)$ and the dual Cheeger constant $\bar{h}$ of a graph $G$ respectively, 
and they related these constants to the gap $2 - \lambda^{\max}$ between $2$ and 
the largest eigenvalue $\lambda^{\max}$ of the normalized Laplacian of a graph $G$.
Bauer and Jost used a technique developed by Desai and Rao \cite{Desai-Rao94} 
in their work. Trevisan proved the following inequality, known as the dual 
Cheeger--Buser inequality, using probabilistic arguments.
\begin{equation*}
    \frac{\beta(G)^2}{2} \leq 2 - \lambda^{\max} \leq 2 \beta(G).
\end{equation*}

Lov\'asz and his collaborators \cites{Lovasz-Szegedy-limitsofDenseGraphs06,
Lovasz-etal-Counthomo06,Lovasz-etal-cgtseqDenseGraphs08} developed 
the theory of graph limits, through both algebraic and analytic perspectives. 
They studied graphons and graphings, which arise as limits of convergent sequences 
of graphs and bounded degree graphs, respectively. 
Several results regarding graphons and graphings are discussed quite extensively 
in the book by Lov\'asz \cite{Lovasz-LargeNet12}. 
The theory of graph limits has found lots of connections with many other branches 
of mathematics, including extremal graph theory, probability theory, higher-order 
Fourier analysis, ergodic theory, number theory, group theory, representation theory, 
category theory, the limit theory of metric spaces, and numerous applications 
in other subjects like computer science, network theory and statistical physics.

Various notions in the context of graphs have been extended to graph limits, 
for instance, homomorphism densities \cite{Lovasz-Szegedy-limitsofDenseGraphs06},
Szemer\'edi's regularity lemma \cite{Lovasz-Szegedy-SzemerdiLemma07},
independent sets, cliques, and colorings \cite{IndSets-cliques-colorings20}, 
and tilings \cite{TilingsGraphons-21}.
Khetan and Mj \cite{Abhishek-Mahan24} established the analogs of the discrete
Cheeger--Buser inequality for graphs in the case of graphons and graphings. 
Given a connected graphon $W$, having Cheeger constant $h_W$ and 
the bottom of the spectrum of its Laplacian $\lambda_W$, they proved that
\begin{equation*}
    \frac{h_W^2}{8} \leq \lambda_W \leq 2 h_W.
\end{equation*}
They also showed that the Cheeger--Buser inequality for regular graphs 
can be recovered from this inequality for graphons. 

We define the \emph{top of the spectrum} of the Laplacian of a graphon in 
\cref{section:prelim}, and the \emph{bipartiteness ratio} of a graphon in 
\cref{section:bpratio}. We prove the analog of the dual Cheeger--Buser inequality 
for graphons in \cref{section:CBIneq}. We establish the following result.

\begin{thm}[The dual Cheeger--Buser inequality for graphons] \label{thm:main}
    Let $W$ be a connected graphon, $\beta_W$ denote its bipartiteness ratio and
    $\lambda_W^{\max}$ denote the top of the spectrum of its Laplacian. 
    Then the following inequality holds.
    \begin{equation*}
        \frac{\beta_W^2}{2} \leq 2 - \lambda_W^{\max} \leq 2 \beta_W.
    \end{equation*}
\end{thm}

We obtain an upper bound for $2 - \lambda_W^{\max}$ in \cref{lemma:buser}, 
using a characterization of $\beta_W$ in terms of functions taking values in $\{-1,0,1\}$, 
extending Trevisan's idea to graphons. Though $\lambda_W^{\max}$ is defined
using $L^2$ functions on $I$, \cref{lemma:lambda-max-Linfty} gives the expression 
for that in terms of only essentially bounded functions on $I$. \cref{lemma:lambda-max-Linfty} 
is inspired from the analogous lemma by Khetan and Mj \cite{Abhishek-Mahan24}*{Lemma 5.4}. 
\cref{lemma:f/g&intf/intg} shows that under suitable conditions, the ratio of 
the integrals of two functions is at least as large as the ratio of those functions 
at some point. In \cref{lemma:main}, we estimate the integrals of certain 
``suitable'' functions so that those estimates combined with \cref{lemma:f/g&intf/intg} 
give an upper bound for $\beta_W$ in terms of $\lambda_W^{\max}$, completing 
the proof of \cref{thm:main}. Both of these lemmas use techniques from 
Trevisan's arguments \cite{Trevisan-notes-expanders}*{Chapter 6}.

As a consequence of \cref{thm:main}, for any connected graphon $W$, we have 
$\lambda_W^{\max}$ is equal to $2$ if and only if $\beta_W$ is equal to $0$. 
In \cref{section:bipartiteGraphons}, we show that this is also equivalent to
the graphon $W$ being bipartite, provided its ``degree function'' is bounded away from zero.

In \cref{section:graphs&graphons}, we compare the largest eigenvalue of the Laplacian 
of a graph with the top of the spectrum of the Laplacian of the associated graphon, 
and the bipartiteness ratio of a graph with the bipartiteness ratio of 
the associated graphon, using ideas in the work of Khetan and Mj \cite{Abhishek-Mahan24}.
Using this comparison, \cref{thm:main} yields the dual Cheeger--Buser inequality
for graphs, up to a multiplicative constant.

\section{Preliminaries} \label{section:prelim}

In the following, by a measurable subset, we mean a Lebesgue measurable subset, 
and we denote the Lebesgue measure on $I = [0,1]$ by $\mu_L$.

A funcion $\map{W}{I^2}{I}$ is called a \emph{graphon} if $W$ is a Lebesgue measurable 
function which is symmetric, that is, $W(x,y) = W(y,x)$ for all $(x,y) \in I^2$. 
We say that a graphon $W$ is \emph{connected} if $\int_{A \times A^c} W > 0$
for every measurable subset $A$ of $I$ with $0 < \mu_L(A) < 1$.

Let $W$ be a connected graphon. For every measurable subset $A$ of $I$ and 
$S$ of $I^2$, define
\begin{equation*}
    \nu(A) \coloneq \int_{A \times I} W(x,y)\, \mathrm{d}x\, \mathrm{d}y, 
    \quad \text{ and } \quad
    \eta(S) \coloneq \int_{S} W(x,y)\, \mathrm{d}x\, \mathrm{d}y.
\end{equation*}
Then $\nu$ and $\eta$ are measures on $I$ and $I^2$ respectively. 
Note that the $\mathbb{R}$-vector space $L^2(I, \nu)$ is a Hilbert space with 
the inner product $\ip{\cdot}{\cdot}_v$, given by
\begin{equation*}
    \ip{f}{g}_v = \int_{I^2} f(x) g(x) W(x,y)\, \mathrm{d}y\, \mathrm{d}x,
\end{equation*}
for all $f,g \in L^2(I, \nu)$. We denote the restriction of the measure $\eta$
to the measurable subsets of the set $E = \{(x,y) \in I^2 : y > x\}$ also by $\eta$, 
and denote the inner product on the $\mathbb{R}$-Hilbert space $L^2(E, \eta)$
by $\ip{\cdot}{\cdot}_e$, which is given by
\begin{equation*}
    \ip{f}{g}_e = \int_0^1 \int_{x}^{1} f(x,y) g(x,y) W(x,y)\, \mathrm{d}y\, \mathrm{d}x,
\end{equation*}
for all $f, g \in L^2(E, \eta)$. The norms induced by the inner products 
$\ip{\cdot}{\cdot}_v$ and $\ip{\cdot}{\cdot}_e$ are denoted by $\norm[v]{\cdot}$
and $\norm[e]{\cdot}$ respectively.

Given any function $f \in L^2(I, \nu)$, define $df(x,y) \coloneq f(y) - f(x)$,
for all $(x,y) \in I^2$. Then it is proved in \cite{Abhishek-Mahan24}*{Lemma 3.3} that
the map $\map{d}{L^2(I, \nu)}{L^2(E, \eta)}$ which maps $f \in L^2(I, \nu)$ to 
the function $df|_E \in L^2(E, \eta)$ is a bounded linear operator. 
Let $\map{d^*}{L^2(E, \eta)}{L^2(I, \nu)}$ denote the adjoint of the operator $d$,
and define the \emph{Laplacian} $\Delta_W$ of $W$ by $\Delta_W = d^*d$, 
which is a bounded linear operator on the space $L^2(I, \nu)$. 

Given a graphon $W$, for all $x \in I$, the \emph{degree} of $x$ is defined by
\begin{equation*}
    d_W(x) \coloneq \int_I W(x,y)\, \mathrm{d}y.
\end{equation*}
If $W$ is a connected graphon, then $\eta(I^2) = \int_I d_W(x)\, \mathrm{d}x$ is positive,
and thus, $d_W$ is positive $\mu_L$-a.e. 
In that case, for every $f \in L^2(I,\nu)$ and $x \in I$ with $d_W(x) \neq 0$,
it is shown in \cite{Abhishek-Mahan24}*{Section 3.2} that
\begin{equation*}
    (\Delta_W f)(x) = f(x) - \frac{1}{d_W(x)} (T_W f)(x),
\end{equation*}
where the linear operator $\map{T_W}{L^2(I,\nu)}{L^2(I,\nu)}$ is defined by
\begin{equation*}
    (T_W f)(x) = \int_I W(x,y) f(y)\, \mathrm{d}y.
\end{equation*}
For the sake of brevity, we will write $\Delta_W = I - \frac{1}{d_W} T_W$,
where $I$ is the identity operator on $L^2(I,\nu)$.

\begin{defn}[Top of the spectrum]
    Given a conncted graphon $W$, the \emph{top of the spectrum} of its Laplacian $\Delta_W$, 
    denoted by $\lambda_W^{\max}$, is defined by
    \begin{equation*}
        \lambda_W^{\max} 
        \coloneq \sup_{f \in L^2(I,\nu) \setminus \{0\}} \frac{\ip{\Delta_W f}{f}_v}{\ip{f}{f}_v}.
    \end{equation*}
\end{defn}

Note that
\begin{equation*}
    \lambda_W^{\max} 
    = \sup_{f \in L^2(I,\nu) \setminus \{0\}} \frac{\norm[e]{df}^2}{\norm[v]{f}^2},
\end{equation*}
and it follows from the proof of \cite{Abhishek-Mahan24}*{Lemma 3.3} that 
$\lambda_W^{\max} \leq 4$. In fact, this bound improves to $2$, 
similar to that in the case of graphs, as shown in the following lemma.

\begin{lemma} \label{lemma:df-leq-2f}
    For any $f \in L^2(I, \nu)$, the inequality 
    $\norm[e]{df} \leq \sqrt{2} \norm[v]{f}$ holds, and consequently, we have
    $\lambda_W^{\max} \leq 2$.
\end{lemma}

\begin{proof}
    Let $f \in L^2(I, \nu)$ be arbitrary. Then we have
    \begin{equation*}
        \norm[e]{df}^2 = \int_{E} (df)^2 W
        = \int_{0}^{1} \int_{x}^{1} (f(y) - f(x))^2 W(x,y)\, \mathrm{d}y\, \mathrm{d}x.
    \end{equation*}

    Using the fact that the graphon $W$ is symmetric, it follows that
    \begin{align*}
        \int_{0}^{1} \int_{x}^{1} (f(y) - f(x))^2 W(x,y)\, \mathrm{d}y\, \mathrm{d}x
        & = 
        \int_{0}^{1} \int_{y}^{1} (f(x) - f(y))^2 W(y,x)\, \mathrm{d}x\, \mathrm{d}y
        \\
        & =
        \int_{0}^{1} \int_{y}^{1} (f(x) - f(y))^2 W(x,y)\, \mathrm{d}x\, \mathrm{d}y
        \\
        & = 
        \int_{E'} (df)^2 W,
    \end{align*}
    where $E' \coloneq \{(x,y) \in I^2 : y < x\}$. Now observe that
    \begin{align*}
        \int_{I^2} (df)^2 W 
        & = 
        \int_{E} (df)^2 W + \int_{E'} (df)^2 W + \int_{\{(x,x) : x \in I\}} (df)^2 W
        \\
        & = 
        \int_{E} (df)^2 W + \int_{E'} (df)^2 W 
        \\
        & =
        2 \norm[e]{df}^2,
    \end{align*}
    and hence, we get
    \begin{align*}
        \norm[e]{df}^2 
        & = 
        \frac{1}{2} \int_0^1 \int_0^1 (f(x) - f(y))^2 W(x,y)\, \mathrm{d}x\, \mathrm{d}y
        \\
        & \leq
        \frac{1}{2} \int_0^1 \int_0^1 f(x)^2 W(x,y)\, \mathrm{d}x\, \mathrm{d}y 
        + \frac{1}{2} \int_0^1 \int_0^1 f(y)^2 W(x,y)\, \mathrm{d}x\, \mathrm{d}y
        \\
        & \quad +
        \int_0^1 \int_0^1 \abs{f(x)} \abs{f(y)} W(x,y)\, \mathrm{d}x\, \mathrm{d}y
        \\
        & =
        \norm[v]{f}^2 
        + \int_0^1 \int_0^1 \abs{f(x)} \abs{f(y)} W(x,y)\, \mathrm{d}x\, \mathrm{d}y.
    \end{align*}
    Now since the function $f$ lies in $L^2(I, \nu)$, the functions 
    \begin{equation*}
        (x,y) \mapsto |f(x)| \sqrt{W(x,y)} \quad \text{and} \quad 
        (x,y) \mapsto |f(y)| \sqrt{W(x,y)},
    \end{equation*}
    defined on $I^2$, are in $L^2(I^2)$. Then the Cauchy--Schwarz inequality implies that
    \begin{align*}
        & \hspace{0.6cm} \int_0^1 \int_0^1 \abs{f(x)} \abs{f(y)} W(x,y)\, \mathrm{d}x\, \mathrm{d}y
        \\
        & \leq 
        \left(\int_0^1 \int_0^1 f(x)^2 W(x,y)\, \mathrm{d}x\, \mathrm{d}y\right)^\frac{1}{2} 
        \left(\int_0^1 \int_0^1 f(y)^2 W(x,y)\, \mathrm{d}x\, \mathrm{d}y\right)^\frac{1}{2}
        \\
        & =
        \norm[v]{f}^2,
    \end{align*}
    from which we conclude that $\norm[e]{df}^2 \leq 2 \norm[v]{f}^2$.
\end{proof}

\section{Bipartiteness ratio of graphons} \label{section:bpratio}

Given any graph with vertex set $V$, a nonempty subset $S$ of $V$ and a bipartition
$\{L,R\}$ of $S$, Trevisan considered the ratio of the number of edges incident on $S$ 
which ``fail to be cut'' by the partition $\{L,R\}$ to the total number of edges
incident on $S$, and defined the bipartiteness ratio of the graph to be 
the minimum of such ratios over all nonempty subsets and their partitions. 
We refer to \cref{section:graphs&graphons} for the precise definition.
We extend this definition to graphons.

\begin{defn}[Bipartiteness ratio]
    The \emph{bipartiteness ratio} of a connected graphon $W$, denoted by $\beta_W$,
    is defined by
    \begin{equation*}
        \beta_W \coloneq \inf_{\substack{L,R \subseteq I \\ \text{measurable}\\
        \mu_L(L \cup R) > 0 \\ L \cap R = \emptyset}} \beta_W(L,R),
    \end{equation*}
    where for every measurable disjoint subsets $L$ and $R$ of $I$ 
    with $\mu_L(L \cup R)> 0$, we have
    \begin{equation*}
        \beta_W(L,R) 
        = \frac{2 \eta(L \times L) + 2 \eta(R \times R) + \eta((L \cup R) \times (L \cup R)^c)}
        {2 \eta((L \cup R) \times I)}.
    \end{equation*}
\end{defn}
Since the graphon $W$ is connected, the above quantity is well-defined.

We will write $\lambda_W^{\max}$, $\beta_W$ and $\beta_W(L,R)$ as $\lambda^{\max}$, 
$\beta(L,R)$ and $\beta$, respectively, when there is no room for confusion.

Khetan and Mj \cite{Abhishek-Mahan24}*{Lemma 3.2} proved that 
the Cheeger constant of any connected graphon is bounded above by $\frac{1}{2}$, 
using the strong mixing property of the doubling map. 
We follow the same arguments to prove that the bipartiteness ratio of connected graphons 
is also bounded above by $\frac{1}{2}$.

For the sake of completeness we define the notion of strong mixing and state 
its characterization that we will use here.

\begin{defn}[Strong mixing]
    Let $(\Omega,\mathcal{A},\mu)$ be a measure space. A measurable function 
    $\map{T}{\Omega}{\Omega}$ is called \emph{strong mixing} if it is a measure preserving 
    transformation, that is, $\mu(A) = \mu(T^{-1}(A))$ for all $A \in \mathcal{A}$, 
    and for all measurable subsets $A$, $B \in \mathcal{A}$, the function $T$ satisfies
    \begin{equation*}
        \lim_{n \to \infty} \mu(T^{-n}(A) \cap B) = \mu(A) \mu(B).
    \end{equation*} 
\end{defn}

It is well known that the doubling map $\map{S}{I}{I}$, defined by
\begin{equation*}
    S(x) \coloneq 
    \begin{cases}
        2x 
        & \text{if }\, 0 \leq x \leq \dfrac{1}{2},
        \\
        2x-1 
        & \text{if }\, \dfrac{1}{2} < x \leq 1,
    \end{cases}
\end{equation*}
is strong mixing. Then it follows that the function $\map{T}{I^2}{I^2}$, 
defined by $T = S \times S$, is also strong mixing. 
Here $I$ and $I^2$ are endowed with the Lebesgue measure.

\begin{lemma}[Strong mixing property]
    Let $(\Omega,\mathcal{A},\mu)$ be a measure space. Then a measure preserving 
    transformation $\map{T}{\Omega}{\Omega}$ is strong mixing if and only if 
    for all $f$, $g \in L^2(\mu)$, we have
    \begin{equation*}
        \lim_{n \to \infty} \int_\Omega (f \circ T^n)g\, \mathrm{d}\mu 
        = \int_\Omega f\, \mathrm{d}\mu \int_\Omega g\, \mathrm{d}\mu.
    \end{equation*}
\end{lemma}

We denote the characteristic function of a set $A$ by $1_A$. 

\begin{lemma} \label{lemma:beta-leq-half}
    For every connected graphon $W$, the inequality $\beta_W \leq \frac{1}{2}$ holds.
\end{lemma}

\begin{proof}
    Let $S$ denote the doubling map on $I$, as defined above, $T$ denote 
    the map $S \times S$, and $\eta_L$ denote the Lebesgue measure on $I^2$.
    Set $L = \left[0, \frac{1}{2}\right]$ and $R = \left(\frac{1}{2}, 1\right]$.
    Using the strong mixing property for $T$, we get
    \begin{equation*}
        \lim_{n \to \infty} \int_{I^2} (1_{L \times L} \circ T^n)W\, \mathrm{d}\eta_L 
        = \left(\int_{I^2} 1_{L \times L}\, \mathrm{d}\eta_L\right)
        \left(\int_{I^2} W\, \mathrm{d}\eta_L\right) = \frac{\eta(I^2)}{4},
    \end{equation*}
    \begin{equation*}
        \lim_{n \to \infty} \int_{I^2} (1_{R \times R} \circ T^n)W\, \mathrm{d}\eta_L 
        = \left(\int_{I^2} 1_{R \times R}\, \mathrm{d}\eta_L\right)
        \left(\int_{I^2} W\, \mathrm{d}\eta_L\right) = \frac{\eta(I^2)}{4},
    \end{equation*}
    \begin{equation*}
        \lim_{n \to \infty} \int_{I^2} (1_{(L \cup R) \times (L \cup R)^c} 
        \circ T^n)W\, \mathrm{d}\eta_L = 0, \tag{since the set $(L \cup R)^c$ is empty}
    \end{equation*}
    and
    \begin{equation*}
        \lim_{n \to \infty} \int_{I^2} (1_{(L \cup R) \times I} \circ T^n)W\, \mathrm{d}\eta_L 
        = \left(\int_{I^2} 1_{(L \cup R) \times I}\, \mathrm{d}\eta_L\right)
        \left(\int_{I^2} W\, \mathrm{d}\eta_L\right) = \eta(I^2).
    \end{equation*}
    For every $n \geq 1$, let $L_n$ and $R_n$ denote the measurable subsets 
    $S^{-n}(L)$ and $S^{-n}(R)$ of $I$, respectively. Since $L$ and $R$ are disjoint, 
    so are the sets $L_n$ and $R_n$ for all $n$. Also, the fact that $S$ is 
    measure preserving ensures that $\mu_L(L_n \cup R_n) > 0$. 
    Further, for all $n \geq 1$ and subsets $A, B$ of $I$, note that 
    $1_{A \times B} \circ T^n = 1_{T^{-n}(A \times B)} = 1_{S^{-n}(A) \times S^{-n}(B)}$. 
    Therefore, it follows that
    \begin{align*}
        & \hspace{0.6cm} \lim_{n \to \infty} \beta_W(L_n,R_n)
        \\
        & = 
        \lim_{n \to \infty} \frac{2 \eta(L_n \times L_n) + 2 \eta(R_n \times R_n) 
        + \eta((L_n \cup R_n) \times (L_n \cup R_n)^c)}{2 \eta((L_n \cup R_n) \times I)}
        \\
        & =
        \lim_{n \to \infty} \frac{2 \int_{I^2} 1_{L_n \times L_n} W\, \mathrm{d}\eta_L 
        + 2 \int_{I^2} 1_{R_n \times R_n} W\, \mathrm{d}\eta_L
        + \int_{I^2} 1_{(L_n \cup R_n) \times (L_n \cup R_n)^c} W\, \mathrm{d}\eta_L}
        {2 \int_{I^2} 1_{(L_n \cup R_n) \times I} W\, \mathrm{d}\eta_L}
        \\
        & =
        \frac{\displaystyle \lim_{n \to \infty} 
        \left[2 \int_{I^2} (1_{L \times L} \circ T^n)W\, \mathrm{d}\eta_L
        + 2 \int_{I^2} (1_{R \times R} \circ T^n)W\, \mathrm{d}\eta_L 
        + \int_{I^2} (1_{(L \cup R) \times (L \cup R)^c} \circ T^n)W\, 
        \mathrm{d}\eta_L\right]}
        {2 \displaystyle \lim_{n \to \infty} 
        \int_{I^2} (1_{(L \cup R) \times I} \circ T^n)W\, \mathrm{d}\eta_L}
        \\
        & =
        \frac{\frac{\eta(I^2)}{2} + \frac{\eta(I^2)}{2}}{2 \eta(I^2)}
        \\
        & =
        \frac{1}{2}.
    \end{align*}
    Now since for all $n$, we have $\beta_W \leq \beta_W(L_n,R_n)$, it follows that
    $\beta_W \leq \frac{1}{2}$.
\end{proof}

\begin{remark}
    In fact, the bound in the above lemma is sharp as can be seen from the following example. 
    If $W$ is a nonzero constant graphon, then for any disjoint measurable subsets
    $L$ and $R$ of $I$ with $\mu_L(L \cup R) > 0$, we have
    \begin{align*}
        \beta_W(L,R) \
        & =
        \frac{2 \eta(L \times L) + 2 \eta(R \times R) + \eta((L \cup R) \times (L \cup R)^c)}
        {2 \eta((L \cup R) \times I)}
        \\
        & =
        \frac{2 \mu_L(L)^2 + 2 \mu_L(R)^2 + (\mu_L(L) + \mu_L(R)) (1 - (\mu_L(L) + \mu_L(R)))}
        {2 (\mu_L(L) + \mu_L(R))}
        \\
        & =
        \frac{1}{2} + \frac{2 \mu_L(L)^2 + 2 \mu_L(R)^2 - (\mu_L(L) + \mu_L(R))^2}
        {2 (\mu_L(L) + \mu_L(R))}
        \\
        & =
        \frac{1}{2} + \frac{(\mu_L(L) - \mu_L(R))^2}{2 (\mu_L(L) + \mu_L(R))}
        \\
        & \geq \frac{1}{2},
    \end{align*} 
    and hence, using \cref{lemma:beta-leq-half}, we conclude that 
    the bipartiteness ratio of $W$ is equal to $\frac{1}{2}$.
\end{remark}

\section{The dual Cheeger--Buser inequality for graphons} \label{section:CBIneq}

\subsection{The dual Buser inequality}

Here we establish an upper bound for $2 - \lambda^{\max}$, by obtaining a characterization 
of $\beta$ in terms of functions taking values in $\{-1,0,1\}$, 
extending Trevisan's idea to graphons.

\begin{lemma} \label{lemma:buser}
    For every connected graphon $W$, the inequality $2 - \lambda^{\max} \leq 2 \beta$ holds.
\end{lemma}

\begin{proof}
    Note that
    \begin{align*}
        2 - \lambda^{\max} 
        & = 
        2 - \sup_{f \in L^2(I, \nu) \setminus \{0\}} \frac{\norm[e]{df}^2}{\norm[v]{f}^2}
        \\
        & =
        2 - \sup_{f \in L^2(I, \nu) \setminus \{0\}} 
        \frac{\int_0^1 \int_0^1 (f(x) - f(y))^2 W(x,y)\, \mathrm{d}x\, \mathrm{d}y}
        {2 \int_0^1 \int_0^1 f(x)^2 W(x,y)\, \mathrm{d}x\, \mathrm{d}y}
        \\
        & =
        \inf_{f \in L^2(I, \nu) \setminus \{0\}} 
        \left(2 - \frac{\int_0^1 \int_0^1 (f(x) - f(y))^2 W(x,y)\, \mathrm{d}x\, \mathrm{d}y}
        {2 \int_0^1 \int_0^1 f(x)^2 W(x,y)\, \mathrm{d}x\, \mathrm{d}y}\right)
        \\
        & =
        \inf_{f \in L^2(I, \nu) \setminus \{0\}}
        \frac{4 \int_0^1 \int_0^1 f(x)^2 W(x,y)\, \mathrm{d}x\, \mathrm{d}y
        - \int_0^1 \int_0^1 (f(x) - f(y))^2 W(x,y)\, \mathrm{d}x\, \mathrm{d}y}
        {2 \int_0^1 \int_0^1 f(x)^2 W(x,y)\, \mathrm{d}x\, \mathrm{d}y}.
    \end{align*}
    The numerator in the above expression is
    \begin{align*}
        & \hspace{0.6cm}
        4 \int_0^1 \int_0^1 f(x)^2 W(x,y)\, \mathrm{d}x\, \mathrm{d}y
        - \int_0^1 \int_0^1 (f(x) - f(y))^2 W(x,y)\, \mathrm{d}x\, \mathrm{d}y
        \\
        & =
        3 \int_0^1 \int_0^1 f(x)^2 W(x,y)\, \mathrm{d}x\, \mathrm{d}y
        - \int_0^1 \int_0^1 f(y)^2 W(x,y)\, \mathrm{d}x\, \mathrm{d}y
        \\
        & \quad +
        2 \int_0^1 \int_0^1 f(x) f(y) W(x,y)\, \mathrm{d}x\, \mathrm{d}y
        \\
        & =
        2 \int_0^1 \int_0^1 f(x)^2 W(x,y)\, \mathrm{d}x\, \mathrm{d}y
        + 2 \int_0^1 \int_0^1 f(x) f(y) W(x,y)\, \mathrm{d}x\, \mathrm{d}y
        \\
        & =
        \int_0^1 \int_0^1 f(x)^2 W(x,y)\, \mathrm{d}x\, \mathrm{d}y 
        + \int_0^1 \int_0^1 f(y)^2 W(x,y)\, \mathrm{d}x\, \mathrm{d}y
        \\
        & \quad +
        2 \int_0^1 \int_0^1 f(x) f(y) W(x,y)\, \mathrm{d}x\, \mathrm{d}y
        \\
        & =
        \int_0^1 \int_0^1 (f(x) + f(y))^2 W(x,y)\, \mathrm{d}x\, \mathrm{d}y.
    \end{align*}
    Therefore, it follows that
    \begin{equation} \label{eq:2-lambdaMax}
        2 - \lambda^{\max} = \inf_{f \in L^2(I, \nu) \setminus \{0\}}
        \frac{\int_0^1 \int_0^1 (f(x) + f(y))^2 W(x,y)\, \mathrm{d}x\, \mathrm{d}y}
        {2 \int_0^1 \int_0^1 f(x)^2 W(x,y)\, \mathrm{d}x\, \mathrm{d}y}.
    \end{equation}
    In order to show that $2 - \lambda^{\max} \leq 2 \beta$, we now proceed 
    to obtain an expression for $\beta$ in terms of functions defined on $I$.
    
    Given any disjoint measurable subsets $L$ and $R$ of $I$ with $\mu_L(L \cup R) > 0$,
    define a function $\map{f}{I}{\{-1,0,1\}}$ for every $x \in I$ as follows.
    \begin{equation*}
        f(x) = 
        \begin{cases}
            -1 & \text{if } x \in L,
            \\
            1 & \text{if } x \in R,
            \\
            0 & \text{if } x \notin L \cup R.
        \end{cases}
    \end{equation*}
    Then $f$ is a nonzero function in $L^2(I, \nu)$, and we have
    \begin{equation} \label{eq:partitions<->fns}
        \beta_W(L,R)
        = \frac{\int_{0}^{1} \int_{0}^{1} (f(x) + f(y))^2 W(x,y)\, \mathrm{d}x\, \mathrm{d}y}
        {4 \int_0^1 \int_0^1 f(x)^2 W(x,y)\, \mathrm{d}x\, \mathrm{d}y}.
    \end{equation}
    To see this, observe that
    \begin{align*}
        & \hspace{0.6cm}
        \int_{0}^{1} \int_{0}^{1} (f(x) + f(y))^2 W(x,y)\, \mathrm{d}x\, \mathrm{d}y 
        \\
        & = 
        \int_{L \cup R} \int_{L \cup R} (f(x) + f(y))^2 W(x,y)\, \mathrm{d}x\, \mathrm{d}y
        + \int_{L \cup R} \int_{(L \cup R)^c} (f(x) + f(y))^2 W(x,y)\, \mathrm{d}x\, \mathrm{d}y
        \\
        & \quad +
        \int_{(L \cup R)^c} \int_{L \cup R} (f(x) + f(y))^2 W(x,y)\, \mathrm{d}x\, \mathrm{d}y
        + \int_{(L \cup R)^c} \int_{(L \cup R)^c} (f(x) + f(y))^2 W(x,y)\, \mathrm{d}x\, \mathrm{d}y
        \\
        & =
        \int_{L \cup R} \int_{L \cup R} (f(x) + f(y))^2 W(x,y)\, \mathrm{d}x\, \mathrm{d}y
        + 2 \eta((L \cup R) \times (L \cup R)^c) 
        \\
        & = 
        \int_L \int_L (f(x) + f(y))^2 W(x,y)\, \mathrm{d}x\, \mathrm{d}y
        + \int_L \int_R (f(x) + f(y))^2 W(x,y)\, \mathrm{d}x\, \mathrm{d}y 
        \\
        & \quad +
        \int_R \int_L (f(x) + f(y))^2 W(x,y)\, \mathrm{d}x\, \mathrm{d}y
        + \int_R \int_R (f(x) + f(y))^2 W(x,y)\, \mathrm{d}x\, \mathrm{d}y 
        \\
        & \quad +
        2 \eta((L \cup R) \times (L \cup R)^c) 
        \\
        & =
        4 \eta(L \times L) + 4 \eta(R \times R) + 2 \eta((L \cup R) \times (L \cup R)^c),
    \end{align*}
    and that
    \begin{align*} 
        & \hspace{0.6cm} 
        \int_0^1 \int_0^1 f(x)^2 W(x,y)\, \mathrm{d}x\, \mathrm{d}y 
        \\
        & =
        \int_{L \cup R} \int_I f(x)^2 W(x,y)\, \mathrm{d}x\, \mathrm{d}y
        + \int_{(L \cup R)^c} \int_I f(x)^2 W(x,y)\, \mathrm{d}x\, \mathrm{d}y
        \\
        & =
        \int_{L \cup R} \int_I W(x,y)\, \mathrm{d}x\, \mathrm{d}y 
        \\
        & =
        \eta((L \cup R) \times I).
    \end{align*}
    On the other hand, given a nonzero function $\map{f}{I}{\{-1,0,1\}}$ in $L^2(I, \nu)$,
    the sets $L = f^{-1}(-1)$ and $R = f^{-1}(1)$ are disjoint measurable subsets
    of $I$ with $\mu_L(L \cup R) > 0$ such that \cref{eq:partitions<->fns} holds. 
    Hence, we conclude that
    \begin{equation} \label{eq:beta-with-fns}
        \beta_W = \inf_{\substack{\map{f}{I}{\{-1,0,1\}} \\ f \in L^2(I, \nu) \setminus \{0\}}}
        \frac{\int_{0}^{1} \int_{0}^{1} (f(x) + f(y))^2 W(x,y)\, \mathrm{d}x\, \mathrm{d}y}
        {4 \int_0^1 \int_0^1 f(x)^2 W(x,y)\, \mathrm{d}x\, \mathrm{d}y}.
    \end{equation}
    Now combine \cref{eq:2-lambdaMax} and \cref{eq:beta-with-fns} to get the inequality 
    $2 - \lambda^{\max} \leq 2 \beta$.
\end{proof}

\subsection{The dual Cheeger inequality}

We obtain a lower bound on $2 - \lambda^{\max}$ with the help of some lemmas. 
The following lemma allows us to work with just essentially bounded functions
instead of all $L^2$ functions while dealing with $\lambda^{\max}$.
It is inspired from the analogous lemma in the work of Khetan and Mj 
\cite{Abhishek-Mahan24}*{Lemma 5.4}.

\begin{lemma} \label{lemma:lambda-max-Linfty}
    Given a connected graphon $W$, we have
    \begin{equation*}
        \lambda^{\max} = \sup_{f \in L^{\infty}(I, \nu) \setminus \{0\}}
        \frac{\norm[e]{df}^2}{\norm[v]{f}^2},
    \end{equation*}
    and consequently, the equality 
    \begin{equation*}
        2 - \lambda^{\max} = \inf_{f \in L^{\infty}(I, \nu) \setminus \{0\}}
        \frac{\int_0^1 \int_0^1 (f(x) + f(y))^2 W(x,y)\, \mathrm{d}x\, \mathrm{d}y}
        {2 \int_0^1 \int_0^1 f(x)^2 W(x,y)\, \mathrm{d}x\, \mathrm{d}y}
    \end{equation*}
    holds.
\end{lemma}

\begin{proof}
    Since the measure space $I$ is $\nu$-finite, it is clear from the inclusion
    $L^{\infty}(I, \nu) \subseteq L^2(I, \nu)$ that $\lambda^{\max}$ is an upper 
    bound of the set 
    \begin{equation*}
        \left\{\frac{\norm[e]{df}^2}{\norm[v]{f}^2} : 
        f \in L^{\infty}(I, \nu) \setminus \{0\}\right\}.
    \end{equation*}
    Let $\varepsilon > 0$ be arbitrary. It is enough to show, for the first part,
    that there is a function $g \in L^{\infty}(I, \nu) \setminus \{0\}$ such that 
    the inequality
    \begin{equation*}
        \lambda^{\max} - \varepsilon < \frac{\norm[e]{dg}^2}{\norm[v]{g}^2}
    \end{equation*} 
    holds. The definition of $\lambda^{\max}$ guarantees the existence of a function
    $f \in L^2(I, \nu) \setminus \{0\}$ with 
    \begin{equation*}
        \lambda^{\max} - \frac{\varepsilon}{2} < \frac{\norm[e]{df}^2}{\norm[v]{f}^2}.
    \end{equation*} 
    So, we are done once we find $g \in L^{\infty}(I, \nu) \setminus \{0\}$ satisfying
    \begin{equation} \label{ineq:f&g-close}
        \frac{\norm[e]{df}^2}{\norm[v]{f}^2} - \frac{\norm[e]{dg}^2}{\norm[v]{g}^2}
        \leq \frac{\varepsilon}{2}.
    \end{equation}
    If the function $df$ is zero, then \eqref{ineq:f&g-close} holds by taking $g$ 
    to be the constant function $1$. Suppose $df$ is nonzero, and define 
    $M = \min\{\norm[v]{f}, \norm[e]{df}\}$, which is a positive real number. 
    As the space $L^{\infty}(I, \nu)$ is dense in $L^2(I, \nu)$, there exists 
    a function $g \in L^{\infty}(I, \nu)$ such that $\norm[v]{g-f} < \varepsilon'M$, 
    where $\varepsilon' = \min\left\{\frac{1}{\sqrt{2}}, \frac{(\sqrt{2} - 1)^2 \varepsilon}
    {16 (\sqrt{2} + 1)}\right\}$. Then, we get the inequality
    \begin{equation} \label{ineq:f&g}
        (1 - \varepsilon')\norm[v]{f} \leq \norm[v]{f} - \varepsilon'M
        < \norm[v]{g} < \norm[v]{f} + \varepsilon'M \leq (1 + \varepsilon')\norm[v]{f}.
    \end{equation}
    This ensures that $g$ is a nonzero function. Since we have 
    $\norm[e]{d(g-f)} < \sqrt{2} \varepsilon'M \leq \sqrt{2} \varepsilon' \norm[e]{df}$ 
    using \cref{lemma:df-leq-2f}, it follows that
    \begin{equation} \label{ineq:df&dg}
        (1 - \sqrt{2} \varepsilon')\norm[e]{df} < \norm[e]{dg} 
        < (1 + \sqrt{2} \varepsilon')\norm[e]{df}.
    \end{equation}
    Using \eqref{ineq:f&g} and the fact that $\varepsilon' \leq 1$, we get
    \begin{equation} \label{ineq:den-lambda-infty}
        \norm[v]{f}^2 \norm[v]{g}^2 \geq (1 - \varepsilon')^2 \norm[v]{f}^4,
    \end{equation}
    and combining \eqref{ineq:f&g}, \eqref{ineq:df&dg} and \cref{lemma:df-leq-2f} 
    gives us that
    \begin{equation} \label{ineq:num-lambda-infty}
        \norm[e]{df}^2 \norm[v]{g}^2 - \norm[e]{dg}^2 \norm[v]{f}^2 \leq ((1 + \varepsilon')^2 
        - (1 - \sqrt{2} \varepsilon')^2) \norm[v]{f}^2 \norm[e]{df}^2.
    \end{equation} 
    Now using \eqref{ineq:den-lambda-infty}, \eqref{ineq:num-lambda-infty} and the fact 
    that $\norm[e]{df}^2 \leq 2\norm[v]{f}^2$ along with the definition of $\varepsilon'$, 
    we arrive at the desired inequality \eqref{ineq:f&g-close}. 

    In order to prove the second part, repeat the arguments used to obtain 
    \cref{eq:2-lambdaMax} by replacing $L^2(I, \nu)$ with $L^{\infty}(I, \nu)$.
\end{proof}

The following two lemmas follow ideas in Trevisan's proof (see
\cite{Trevisan-MaxCut12}*{Section 3.2} and \cite{Trevisan-notes-expanders}*{Chapter 6}), 
of the analogous inequality in the case of graphs. 

\begin{lemma} \label{lemma:f/g&intf/intg}
    Let $(\Omega, \mathcal{A}, \mu)$ be a measure space with $\mu(\Omega) > 0$,
    and $\map{f}{\Omega}{\mathbb{R}}$, $\map{g}{\Omega}{(0,\infty)}$ be integrable
    functions. Then there exists an $x_0 \in \Omega$ such that 
    \begin{equation*}
        \frac{f(x_0)}{g(x_0)} \leq \frac{\int_{\Omega} f(x)\, \mathrm{d}\mu(x)}
        {\int_{\Omega} g(x)\, \mathrm{d}\mu(x)}.
    \end{equation*}
\end{lemma}

\begin{proof}
    Since the function $g$ takes only positive values and $\mu(\Omega) > 0$,
    we have $\int_{\Omega} g(x)\, \mathrm{d}\mu(x) > 0$. Further, note that 
    \begin{equation} \label{eq:int0}
        \int_{\Omega} \left(f(x) - \frac{\int_{\Omega} f(t)\, \mathrm{d}\mu(t)}
        {\int_{\Omega} g(t)\, \mathrm{d}\mu(t)} g(x) \right) \mathrm{d}\mu(x) = 0.
    \end{equation}
    If $f(x) - \frac{\int_{\Omega} f(t)\, \mathrm{d}\mu(x)}
    {\int_{\Omega} g(t)\, \mathrm{d}\mu(t)}g(x)$ is positive for all $x$, then 
    its integration over $\Omega$ is positive, as $\mu(\Omega)$ is positive,
    which contradicts \cref{eq:int0}. Hence, there is an $x_0 \in \Omega$ such that 
    \begin{equation*}
        f(x_0) - \frac{\int_{\Omega} f(x_)\, \mathrm{d}\mu(x)}
        {\int_{\Omega} g(x)\, \mathrm{d}\mu(x)}g(x_0) \leq 0,
    \end{equation*}
    that is
    \begin{equation*}
        \frac{f(x_0)}{g(x_0)} \leq \frac{\int_{\Omega} f(x)\, \mathrm{d}\mu(x)}
        {\int_{\Omega} g(x)\, \mathrm{d}\mu(x)}.
    \end{equation*}
\end{proof}

The above lemma shows that under suitable conditions, the ratio of the integrals
of two functions is at least as large as the ratio of those functions at some point.
In the following lemma, we estimate the integrals of certain ``suitable'' functions 
so that those estimates combined with \cref{lemma:f/g&intf/intg} give 
an upper bound for $\beta$ in terms of $\lambda^{\max}$.

\begin{lemma} \label{lemma:main}
    Let $f$ be an arbitrary element of $L^2(I, \nu)$. For every $t > 0$, the sets 
    defined by $L_t = f^{-1}((-\infty, -t])$ and $R_t = f^{-1}([t, \infty))$ 
    have the properties
    \begin{align} \label{ineq:int-num}
        & \hspace{0.6cm}
        \int_{0}^{\infty} 2t \left[2 \eta(L_t \times L_t) + 2 \eta(R_t \times R_t)
        + \eta((L_t \cup R_t) \times (L_t \cup R_t)^c)\right] \mathrm{d}t \nonumber
        \\
        & \leq
        2 \left(\int_{0}^{1} \int_{0}^{1} (f(x) + f(y))^2 W(x,y)\, \mathrm{d}x\, \mathrm{d}y
        \right)^\frac{1}{2} \left(\int_{0}^{1} \int_{0}^{1} f(x)^2 W(x,y)\, 
        \mathrm{d}x\, \mathrm{d}y\right)^\frac{1}{2},
    \end{align}
    and
    \begin{equation} \label{eq:int-den}
        \int_{0}^{\infty} 2t \left[2 \eta((L_t \cup R_t) \times I)\right] \mathrm{d}t
        = 2 \int_{0}^{1} \int_{0}^{1} f(x)^2 W(x,y)\, \mathrm{d}x\, \mathrm{d}y.
    \end{equation}
\end{lemma}

\begin{proof}
    Note that
    \begin{align*}
        & \hspace{0.6cm}
        \int_{0}^{\infty} 2t \left[2 \eta(L_t \times L_t) + 2 \eta(R_t \times R_t)
        + \eta((L_t \cup R_t) \times (L_t \cup R_t)^c)\right] \mathrm{d}t 
        \\
        & = 
        \int_{0}^{\infty} 2t \left[\int_0^1 \int_{0}^{1} \left(2 \cdot 1_{L_t \times L_t}
        + 2 \cdot 1_{R_t \times R_t} + 1_{(L_t \cup R_t) \times (L_t \cup R_t)^c}\right) 
        W(x,y)\, \mathrm{d}x\, \mathrm{d}y\right] \mathrm{d}t
        \\
        & = 
        \int_{I^2} \left[\int_{0}^{\infty} 2t 
        (2 \cdot 1_{L_t \times L_t}(x,y) + 2 \cdot 1_{R_t \times R_t}(x,y) 
        1_{(L_t \cup R_t) \times (L_t \cup R_t)^c}(x,y))\, \mathrm{d}t\right]
        W(x,y)\, \mathrm{d}x\, \mathrm{d}y. \tag{using the Fubini--Tonelli theorem}
    \end{align*}
    Define the sets
    \begin{align*}
        A_1 & = \{(x,y) \in I^2 : 0 \leq f(x) \leq f(y)\},
        \\
        A_2 & = \{(x,y) \in I^2 : 0 \leq f(y) < f(x)\},
        \\
        A_3 & = \{(x,y) \in I^2 : f(x) < f(y) \leq 0\},
        \\
        A_4 & = \{(x,y) \in I^2 : f(y) \leq f(x) \leq 0\},
        \\
        A_5 & = \{(x,y) \in I^2 : f(x)f(y) < 0, \abs{f(y)} < \abs{f(x)}\}.
    \end{align*}
    Now given any $(x,y) \in I^2$, observe that
    \begin{equation*}
        1_{L_t \times L_t}(x,y) =
        \begin{cases}
            1 & \text{if } (x,y) \in A_3 \text{ and } t \in (0, -f(y)],
            \\
            1 & \text{if } (x,y) \in A_4 \text{ and } t \in (0, -f(x)],
            \\
            0 & \text{otherwise,}
        \end{cases} 
    \end{equation*}
    \begin{equation*}
        1_{R_t \times R_t}(x,y) =
        \begin{cases}
            1 & \text{if } (x,y) \in A_1 \text{ and } t \in (0, f(x)],
            \\
            1 & \text{if } (x,y) \in A_2 \text{ and } t \in (0, f(y)],
            \\
            0 & \text{otherwise,}
        \end{cases} 
    \end{equation*}
    and that
    \begin{equation*}
        1_{(L_t \cup R_t) \times (L_t \cup R_t)^c}(x,y) =
        \begin{cases}
            1 & \text{if } (x,y) \in A_2 \text{ and } t \in (f(y), f(x)],
            \\
            1 & \text{if } (x,y) \in A_3 \text{ and } t \in (-f(y), -f(x)],
            \\
            1 & \text{if } (x,y) \in A_5 \text{ and } t \in (\abs{f(y)}, \abs{f(x)}],
            \\
            0 & \text{otherwise.}
        \end{cases} 
    \end{equation*}
    Thus, we get
    \begin{align*}
        & \hspace{0.6cm} \int_0^1 \int_{0}^{1} 
        \left[\int_{0}^{\infty} 2t \left(2 \cdot 1_{L_t \times L_t}(x,y)
        \right) \mathrm{d}t\right] W(x,y)\, \mathrm{d}x\, \mathrm{d}y
        \\
        & = 
        \int_{A_3} \left(\int_{0}^{-f(y)} 4t\, \mathrm{d}t\right) W(x,y)\, 
        \mathrm{d}x\, \mathrm{d}y + \int_{A_4} \left(\int_{0}^{-f(x)} 4t\, 
        \mathrm{d}t\right) W(x,y)\, \mathrm{d}x\, \mathrm{d}y
        \\
        & =
        \int_{A_3} 2f(y)^2 W(x,y)\, \mathrm{d}x\, \mathrm{d}y 
        + \int_{A_4} 2f(x)^2 W(x,y)\, \mathrm{d}x\, \mathrm{d}y,
    \end{align*}
    and
    \begin{align*}
        & \hspace{0.6cm} \int_0^1 \int_{0}^{1} 
        \left[\int_{0}^{\infty} 2t \left(2 \cdot 1_{R_t \times R_t}(x,y)
        \right) \mathrm{d}t\right] W(x,y)\, \mathrm{d}x\, \mathrm{d}y
        \\
        & = 
        \int_{A_1} \left(\int_{0}^{f(x)} 4t\, \mathrm{d}t\right) W(x,y)\, 
        \mathrm{d}x\, \mathrm{d}y + \int_{A_2} \left(\int_{0}^{f(y)} 4t\, 
        \mathrm{d}t\right) W(x,y)\, \mathrm{d}x\, \mathrm{d}y
        \\
        & =
        \int_{A_1} 2f(x)^2 W(x,y)\, \mathrm{d}x\, \mathrm{d}y 
        + \int_{A_2} 2f(y)^2 W(x,y)\, \mathrm{d}x\, \mathrm{d}y,
    \end{align*}
    and 
    \begin{align*}
        & \hspace{0.6cm} \int_0^1 \int_{0}^{1} 
        \left[\int_{0}^{\infty} 2t \left(1_{(L_t \cup R_t) \times (L_t \cup R_t)^c}
        (x,y)\right) \mathrm{d}t\right] W(x,y)\, \mathrm{d}x\, \mathrm{d}y
        \\
        & = 
        \int_{A_2} \left(\int_{f(y)}^{f(x)} 2t\, \mathrm{d}t\right) W(x,y)\, 
        \mathrm{d}x\, \mathrm{d}y + \int_{A_3} \left(\int_{-f(y)}^{-f(x)} 2t\, 
        \mathrm{d}t\right) W(x,y)\, \mathrm{d}x\, \mathrm{d}y 
        \\
        & \quad +
        \int_{A_5} \left(\int_{\abs{f(y)}}^{\abs{f(x)}} 2t\, \mathrm{d}t\right)
        W(x,y)\, \mathrm{d}x\, \mathrm{d}y
        \\
        & =
        \int_{A_2} (f(x)^2 - f(y)^2) W(x,y)\, \mathrm{d}x\, \mathrm{d}y 
        + \int_{A_3} (f(x)^2 - f(y)^2) W(x,y)\, \mathrm{d}x\, \mathrm{d}y
        \\
        & \quad +
        \int_{A_5} (f(x)^2 - f(y)^2) W(x,y)\, \mathrm{d}x\, \mathrm{d}y.
    \end{align*}
    Therefore, we finally have
    \begin{align*}
        & \hspace{0.6cm}
        \int_{0}^{\infty} 2t \left[2 \eta(L_t \times L_t) + 2 \eta(R_t \times R_t)
        + \eta((L_t \cup R_t) \times (L_t \cup R_t)^c)\right] \mathrm{d}t 
        \\
        & =
        \int_{A_1} 2f(x)^2 W(x,y)\, \mathrm{d}x\, \mathrm{d}y 
        + \int_{A_2} (f(x)^2 + f(y)^2) W(x,y)\, \mathrm{d}x\, \mathrm{d}y
        \\
        & \quad +
        \int_{A_3} (f(x)^2 + f(y)^2) W(x,y)\, \mathrm{d}x\, \mathrm{d}y
        + \int_{A_4} 2f(x)^2 W(x,y)\, \mathrm{d}x\, \mathrm{d}y
        \\
        & \quad +
        \int_{A_5} (f(x)^2 - f(y)^2) W(x,y)\, \mathrm{d}x\, \mathrm{d}y,
    \end{align*}
    which is finite, as the function $f$ lies in $L^2(I, \nu)$.  
    Now since each of the above integrands is less than or equal to $\abs{f(x) + f(y)}
    (\abs{f(x)} + \abs{f(y)}) W(x,y)$, it follows that
    \begin{align*}
        & \hspace{0.6cm}
        \int_{0}^{\infty} 2t \left[2 \eta(L_t \times L_t) + 2 \eta(R_t \times R_t)
        + \eta((L_t \cup R_t) \times (L_t \cup R_t)^c)\right] \mathrm{d}t 
        \\
        & \leq
        \int_{0}^{1} \int_{0}^{1} \abs{f(x) + f(y)}(\abs{f(x)} + \abs{f(y)})
        W(x,y)\, \mathrm{d}x\, \mathrm{d}y
        \\
        & \leq 
        \left(\int_{0}^{1} \int_{0}^{1} (f(x) + f(y))^2 W(x,y)\, \mathrm{d}x\, 
        \mathrm{d}y\right)^\frac{1}{2} 
        \left(\int_{0}^{1} \int_{0}^{1} (\abs{f(x)} + \abs{f(y)})^2
        W(x,y)\, \mathrm{d}x\, \mathrm{d}y\right)^\frac{1}{2}
        \tag{using the Cauchy--Schwarz inequality in $L^2(I^2)$}
        \\
        & \leq 
        \left(\int_{0}^{1} \int_{0}^{1} (f(x) + f(y))^2 W(x,y)\, \mathrm{d}x\, 
        \mathrm{d}y\right)^\frac{1}{2} 
        \left(\int_{0}^{1} \int_{0}^{1} (2f(x)^2 + 2f(y)^2)
        W(x,y)\, \mathrm{d}x\, \mathrm{d}y\right)^\frac{1}{2}
        \tag{for real numbers $a, b$, $(\abs{a} + \abs{b})^2 \leq 2a^2 + 2b^2$}
        \\
        & =
        2 \left(\int_{0}^{1} \int_{0}^{1} (f(x) + f(y))^2 W(x,y)\, \mathrm{d}x\, \mathrm{d}y
        \right)^\frac{1}{2} \left(\int_{0}^{1} \int_{0}^{1} f(x)^2 W(x,y)\, 
        \mathrm{d}x\, \mathrm{d}y\right)^\frac{1}{2}.
    \end{align*}
    Similar calculations give us
    \begin{align*}
        \int_{0}^{\infty} 2t \left[2 \eta((L_t \cup R_t) \times I)\right] \mathrm{d}t
        & =
        \int_{0}^{1} \int_{0}^{1} \left(\int_{0}^{\infty} 4t \cdot 1_{(L_t \cup R_t)
        \times I}(x,y)\, \mathrm{d}t\right) W(x,y)\, \mathrm{d}x\, \mathrm{d}y
        \\
        & =
        \int_{0}^{1} \int_{0}^{1} \left(\int_{0}^{\abs{f(x)}} 4t\, \mathrm{d}t\right) 
        W(x,y)\, \mathrm{d}x\, \mathrm{d}y
        \\
        & =
        2 \int_{0}^{1} \int_{0}^{1} f(x)^2 W(x,y)\, \mathrm{d}x\, \mathrm{d}y,
    \end{align*}
    as desired.
\end{proof}

Note that the measures $\mu_L$ and $\nu$ on $I$ are absolutely continuous with
respect to each other, and hence we have $L^{\infty}(I, \mu_L) = L^{\infty}(I, \nu)$.
Let us denote these spaces simply by $L^{\infty}(I)$. Also, observe that 
if $f$ lies in $L^{\infty}(I)$, then its essential supremum with respect to 
both the measures is the same. We denote it by $\norm[\infty]{f}$.

\begin{proof}[Proof of \cref{thm:main}]
    We have already proved one of the inequalities in \cref{lemma:buser}. 
    For the other inequality, thanks to \cref{lemma:lambda-max-Linfty}, it suffices 
    to show that for every nonzero function $f \in L^{\infty}(I)$, there exist 
    disjoint measurable subsets $L$ and $R$ of $I$ with $\mu_L(L \cup R) > 0$ such that 
    the following inequality holds.
    \begin{equation*}
        \beta(L,R) 
        \leq \left(\frac{\int_0^1 \int_0^1 (f(x) + f(y))^2 W(x,y)\, \mathrm{d}x\, \mathrm{d}y}
        {\int_0^1 \int_0^1 f(x)^2 W(x,y)\, \mathrm{d}x\, \mathrm{d}y}\right)^\frac{1}{2}.
    \end{equation*}

    Let $f$ be any nonzero function in $L^{\infty}(I)$. 
    For every $t \in (0, \norm[\infty]{f})$, the sets $L_t$ and $R_t$, as defined
    in \cref{lemma:main}, are disjoint measurable subsets of $I$ with 
    $\mu_L(L_t \cup R_t) > 0$. Also, for $t > \norm[\infty]{f}$, the sets 
    $L_t$ and $R_t$ have measure zero. This implies that
    \begin{align*}
        & \hspace{0.6cm}
        \int_{0}^{\infty} 2t \left[2 \eta(L_t \times L_t) + 2 \eta(R_t \times R_t)
        + \eta((L_t \cup R_t) \times (L_t \cup R_t)^c)\right] \mathrm{d}t
        \\
        & = 
        \int_{0}^{\norm[\infty]{f}} 2t \left[2 \eta(L_t \times L_t) + 2 \eta(R_t \times R_t)
        + \eta((L_t \cup R_t) \times (L_t \cup R_t)^c)\right] \mathrm{d}t
    \end{align*}
    and that 
    \begin{equation*}
        \int_{0}^{\infty} 2t \left[2 \eta((L_t \cup R_t) \times I)\right] \mathrm{d}t
        = \int_{0}^{\norm[\infty]{f}} 2t \left[2 \eta((L_t \cup R_t) \times I)\right] \mathrm{d}t,
    \end{equation*}
    where the integrand $4t \eta((L_t \cup R_t) \times I)$ is positive for 
    every $t \in (0, \norm[\infty]{f})$. Hence, using \cref{lemma:main}, we arrive
    at the inequality
    \begin{align*}
        & \hspace{0.6cm} \frac{\int_{0}^{\norm[\infty]{f}} 2t \left[2 \eta(L_t \times L_t)
        + 2 \eta(R_t \times R_t) + \eta((L_t \cup R_t) \times (L_t \cup R_t)^c)\right] 
        \mathrm{d}t}{\int_{0}^{\norm[\infty]{f}} 2t \left[2 \eta((L_t \cup R_t) 
        \times I)\right] \mathrm{d}t}
        \\
        & \leq
        \frac{2 \left(\int_{0}^{1} \int_{0}^{1} (f(x) + f(y))^2 W(x,y)\, \mathrm{d}x\, 
        \mathrm{d}y\right)^\frac{1}{2} \left(\int_{0}^{1} \int_{0}^{1} f(x)^2 W(x,y)\, 
        \mathrm{d}x\, \mathrm{d}y\right)^\frac{1}{2}}
        {2 \int_{0}^{1} \int_{0}^{1} f(x)^2 W(x,y)\, \mathrm{d}x\, \mathrm{d}y}
        \\
        & =
        \left(\frac{\int_0^1 \int_0^1 (f(x) + f(y))^2 W(x,y)\, \mathrm{d}x\, \mathrm{d}y}
        {\int_0^1 \int_0^1 f(x)^2 W(x,y)\, \mathrm{d}x\, \mathrm{d}y}\right)^\frac{1}{2}.
    \end{align*}
    Now \cref{lemma:f/g&intf/intg} guarantees that there is a 
    $t_0 \in (0, \norm[\infty]{f})$ such that 
    \begin{equation*}
        \beta(L_{t_0},R_{t_0}) \leq 
        \left(\frac{\int_0^1 \int_0^1 (f(x) + f(y))^2 W(x,y)\, \mathrm{d}x\, \mathrm{d}y}
        {\int_0^1 \int_0^1 f(x)^2 W(x,y)\, \mathrm{d}x\, \mathrm{d}y}\right)^\frac{1}{2},
    \end{equation*}
    which completes the proof.
\end{proof}

\section{Bipartite graphons} \label{section:bipartiteGraphons}

Khetan and Mj \cite{Abhishek-Mahan24}*{Section 7.3} gave necessary and sufficient
conditions for a graphon to be conncted, under some suitable hypothesis.
In this section, we characterize bipartite graphons in terms of the top of the spectrum
of their Laplacians and their bipartiteness ratios, under the same hypothesis. 
We start by recalling the definition of bipartite graphons.

\begin{defn}[Bipartite graphon]
    A graphon $W$ is said to be \emph{bipartite} if there exist disjoint measurable subsets
    $L$ and $R$ of $I$ such that $L \cup R = I$ and $W$ is zero almost everywhere
    on $L \times L$ and $R \times R$ (with respect to the Lebesgue measure on $I^2$).
\end{defn}

We will use \cite{Abhishek-Mahan24}*{Lemma 7.11} which states that if $W$ is
a connected graphon and the function $d_W$ is bounded below by a positive real number,
then the operator $\map{\frac{1}{d_W} T_W}{L^2(I,\nu)}{L^2(I,\nu)}$ is compact.

\begin{lemma} \label{lemma:actualEigenVal}
    Let $W$ be a connected graphon such that $d_W$ is bounded below by a positive real number.
    Then $\lambda_W^{\max}$ is an eigenvalue of the Laplacian $\Delta_W$ of $W$.
\end{lemma}

\begin{proof}
    Since the Laplacian of $W$ is a self-adjoint bounded linear operator on the Hilbert space 
    $L^2(I,\nu)$, its top of the spectrum $\lambda_W^{\max}$ is its approximate eigenvalue,
    using \cite{BVLimaye-FA96}*{Theorem 27.5(a)}. Then $1 - \lambda_W^{\max}$ 
    is an approximate eigenvalue of the operator 
    $I - \Delta_W = I - \left(I - \frac{1}{d_W} T_W\right) = \frac{1}{d_W} T_W$.
    Note that $\frac{1}{d_W} T_W$ is a compact operator by \cite{Abhishek-Mahan24}*{Lemma 7.11}.
    We know that every nonzero approximate eigenvalue of a compact operator on a Hilbert space 
    is its eigenvalue (see \cite{BVLimaye-FA96}*{Lemma 28.4(a)} for instance). 
    Hence, in our case, $1 - \lambda_W^{\max}$ is an eigenvalue of $\frac{1}{d_W} T_W$.
    Then it follows that $\lambda_W^{\max}$ is an eigenvalue of $\Delta_W$.
\end{proof}

The following lemma characterizes bipartite graphons.

\begin{lemma} \label{lemma:bipartiteGraphonChar}
    Let $W$ be a connected graphon such that $d_W$ is bounded below by a positive real number.
    Then the following statements are equivalent.
    \begin{enumerate}
        \item $\beta_W = 0$.
        \item $\lambda_W^{\max} = 2$.
        \item The graphon $W$ is bipartite.
    \end{enumerate}
\end{lemma}

\begin{proof}
    The implication $(1) \implies (2)$ follows from the dual Buser inequality 
    (\cref{lemma:buser}), and the implication $(3) \implies (1)$ is a direct
    consequence of the definitions of $\beta_W$ and bipartite graphons.
    Now we prove that $(2) \implies (3)$.

    Suppose that $\lambda_W^{\max} = 2$. Then \cref{lemma:actualEigenVal} ensures that 
    $2$ is an eigenvalue of $\Delta_W$. Let $f \in L^2(I,\nu)$ be 
    its corresponding eigenfunction. Then the arguments similar to those used
    to prove \cref{eq:2-lambdaMax} yield the equation
    \begin{equation*}
        \frac{\int_{I^2} (f(x) + f(y))^2 W(x,y)\, \mathrm{d}x\, \mathrm{d}y}
        {2 \int_{I^2} f(x)^2 W(x,y)\, \mathrm{d}x\, \mathrm{d}y} = 0,
    \end{equation*}
    and hence, 
    \begin{equation} \label{eq:num2-lambda}
        \int_{I^2} (f(x) + f(y))^2 W(x,y)\, \mathrm{d}x\, \mathrm{d}y = 0.
    \end{equation}
    Denote the sets $f^{-1}(-\infty,0)$ and $f^{-1}(0,\infty)$ by $L$ and $R$ respectively.
    Then \cref{eq:num2-lambda} gives that 
    \begin{equation*}
        \int_{(L \cup R) \times (L \cup R)^c} 
        (f(x) + f(y))^2 W(x,y)\, \mathrm{d}x\, \mathrm{d}y = 0.
    \end{equation*}
    For any $(x,y) \in (L \cup R) \times (L \cup R)^c$, since we have 
    $(f(x) + f(y))^2 = f(x)^2 > 0$, it follows that $W$ is zero almost everywhere
    on $(L \cup R) \times (L \cup R)^c$. Now since $W$ is connected, this implies that
    the Lebesgue measure of either $L \cup R$ or its complement is zero. But the fact 
    that the function $f$ is nonzero forces $(L \cup R)^c$ to have measure zero.
    Let $L'$ denote the set $L \cup (L \cup R)^c$. Note that $L'$ and $R$ are
    disjoint measurable subsets of $I$, and their union is $I$. We are done
    once we show that $W$ is zero almost everywhere on $L' \times L'$ and $R \times R$.
    It follows from \cref{eq:num2-lambda} that 
    \begin{equation*}
        \int_{L \times L} (f(x) + f(y))^2 W(x,y)\, \mathrm{d}x\, \mathrm{d}y = 0
        \quad \text{and} \quad
        \int_{R \times R} (f(x) + f(y))^2 W(x,y)\, \mathrm{d}x\, \mathrm{d}y = 0.
    \end{equation*}
    For all $(x,y) \in L \times L$, the quantity $(f(x) + f(y))^2$ is positive, 
    and therefore, $W$ is zero almost everywhere on $L \times L$, and hence 
    also on $L' \times L'$, as $(L \cup R)^c$ has measure zero. 
    Similarly, it follows that $W$ is zero almost everywhere on $R \times R$.
\end{proof}

\section{Graphs and the associated graphons} \label{section:graphs&graphons}

Let $V$ denote the set $\{1, \dots, n\}$ with $n \geq 2$, and $\map{w}{V \times V}
{I}$ be a symmetric function, that is, $w(i,j) = w(j,i)$ for all $i,j \in V$.
The pair $G = (V,w)$ is called a \emph{weighted graph}. We will denote $w(i,j)$ 
by $w_{ij}$, for all $i,j \in V$. The weighted graph $G$ is said to be \emph{loopless}
if $w_{ii} = 0$ for all $i \in V$. For any subsets $A, B$ of $V$, define
\begin{equation*}
    e_G(A,B) = \sum_{i \in A, j \in B} w_{ij}.
\end{equation*}
In the following, we always assume that $G$ is \emph{connected}, which means that
for any nonempty proper subset $A$ of $V$, $e_G(A,A^c)$ is positive. 
For any $i \in V$, the \emph{volume} of $i$ is defined by $\vol(i) = \sum_{j \in V} w_{ij}$. 
Note that $\vol(i)$ is positive for all $i$, since the graph $G$ is connected.

Let $\ell^2(V)$ denote the Hilbert space 
of all functions from $V$ to $\mathbb{R}$, equipped with the inner product
\begin{equation*}
    \ip{f_1}{f_2} \coloneq \sum_{i \in V} f_1(i) f_2(i),
\end{equation*}
for every $f_1,f_2 \in \ell^2(V)$. 
The \emph{Laplacian} $\map{\Delta_G}{\ell^2(V)}{\ell^2(V)}$ of the weighted graph
$G$ is a linear operator defined by
\begin{equation*}
    (\Delta_G g)(i) \coloneq 
    g(i) - \frac{1}{\sqrt{\vol(i)}} \sum_{j \in V} \frac{g(j) w_{ij}}{\sqrt{\vol(j)}},
\end{equation*}
for every $g \in \ell^2(V)$ and $i \in V$. It is a self-adjoint operator, since 
the function $w$ is symmetric. Then the largest eigenvalue $\lambda_G^{\max}$ 
of the Laplacian $\Delta_G$ is given by
\begin{equation*}
    \lambda_G^{\max} 
    = \sup_{g \in \ell^2(V) \setminus \{0\}} \frac{\ip{\Delta_G g}{g}}{\ip{g}{g}}
    = \sup_{g \in \ell^2(V) \setminus \{0\}} 
    \frac{\ip{\Delta_G (\sqrt{D} g)}{\sqrt{D} g}}{\ip{\sqrt{D} g}{\sqrt{D} g}},
\end{equation*}
where $\sqrt{D}$ is an invertible operator on $\ell^2(V)$ defined by
$(\sqrt{D} h)(i) = \sqrt{\vol(i)} h(i)$, for all $h \in \ell^2(V)$ and $i \in V$.
It is easy to check that
\begin{equation*}
    \lambda_G^{\max} 
    = \sup_{g \in \ell^2(V) \setminus \{0\}} 
    \frac{\sum_{i,j \in V} (g(i) - g(j))^2 w_{ij}}{2 \sum_{i,j \in V} g(i)^2 w_{ij}}.
\end{equation*}

The \emph{bipartiteness ratio} $\beta_G$ of the weighted graph $G$ is defined as follows.
\begin{equation*}
    \beta_G = \min_{\substack{A,B \subseteq V \\ A \cup B \neq \emptyset \\
    A \cap B = \emptyset}} \frac{2 e_G(A,A) + 2 e_G(B,B) + e_G(A \cup B, (A \cup B)^c)}
    {2 e_G(A \cup B,V)}.
\end{equation*}

Now given a weighted graph $G = (V, w)$, it can be viewed as the graphon, 
called the \emph{associated graphon} $W_G$ of $G$, defined as below. 
For each $1 \leq i < n$, denote the interval $[\frac{i-1}{n}, \frac{i}{n})$ by $P_i$, 
and $[\frac{n-1}{n}, 1]$ by $P_n$. Note that $\{P_i \times P_j : 1 \leq i,j \leq n\}$ 
forms a partition of $I^2$. For any $1 \leq i,j \leq n$ and $(x,y) \in P_i \times P_j$, 
define $W_G(x,y) \coloneq w_{ij}$.

We will show that the connectedness of $G$ implies the connectedness of $W_G$,
so that we can talk about the Laplacian and the bipartiteness ratio of $W_G$.

\begin{lemma}
    If $G = (V, w)$ is a connected weighted graph, then the associated graphon $W_G$ 
    of $G$ is also connected.
\end{lemma}

\begin{proof}
    Let $A$ be a measurable subset of $I$ with $0 < \mu_L(A) < 1$. Then the sets
    \begin{equation*}
        S_1 = \{i \in V : \mu_L(A \cap P_i) > 0\} \quad \text{and} \quad 
        S_2 = \{j \in V : \mu_L(A^c \cap P_j) > 0\}
    \end{equation*}
    are nonempty, and the inclusions $S_1^c \subseteq S_2$ and $S_2^c \subseteq S_1$
    hold. Further, observe that
    \begin{align*}
        \int_{A \times A^c} W 
        = \sum_{i,j \in V} \int_{(A \cap P_i) \times (A^c \cap P_j)} W 
        & = 
        \sum_{i \in S_1, j \in S_2} \int_{(A \cap P_i) \times (A^c \cap P_j)} w_{ij}
        \\
        & = 
        \sum_{i \in S_1, j \in S_2} \mu_L(A \cap P_i) \mu_L(A^c \cap P_j) w_{ij}
        \\
        & \geq
        m_A e_G(S_1,S_2),
    \end{align*} 
    where $m_A = \min\{\mu_L(A \cap P_i) \mu_L(A^c \cap P_j) : i \in S_1, j \in S_2\} > 0$.
    Now if both $S_1$ and $S_2$ are equal to $V$, then we have $e_G(S_1,S_2) \geq 
    e_G(\{1\},\{1\}^c) > 0$ as the graph $G$ is connected. Otherwise,
    if $S_1$ or $S_2$ is not $V$, then we get $e_G(S_1,S_2) \geq e_G(S_1,S_1^c) > 0$
    or $e_G(S_1,S_2) \geq e_G(S_2^c,S_2) > 0$, respectively. Thus, in any case,
    $m_A e_G(S_1,S_2)$ and hence $\int_{A \times A^c} W$ is positive, showing that
    the associated graphon $W_G$ is connected. 
\end{proof}

\subsection{Top of the spectrum of graphs and the associated graphons}

The arguments in the following lemma are similar to that in 
\cite{Abhishek-Mahan24}*{Section 4.2}.

\begin{lemma} \label{lemma:compare-lambdaG&WG}
    Given any loopless, connected weighted graph $G = (V,w)$, we have
    $\lambda_{W_G}^{\max} = \lambda_G^{\max}$.
\end{lemma}

\begin{proof}
    Let $\map{g}{V}{\mathbb{R}}$ be any nonzero function. It gives rise to 
    a nonzero function $g' \in L^{\infty}(I)$, defined for any $x \in P_i$
    with $1 \leq i \leq n$, by $g'(x) = g(i)$, that satisfies
    \begin{align*}
        \lambda_{W_G}^{\max}
        & \geq 
        \frac{\int_{0}^{1} \int_{0}^{1} (g'(x) - g'(y))^2 W_G(x,y)\, \mathrm{d}x\, \mathrm{d}y}
        {2 \int_{0}^{1} \int_{0}^{1} g'(x)^2 W_G(x,y)\,\mathrm{d}x\, \mathrm{d}y}
        \\
        & =
        \frac{\sum_{i,j \in V} \int_{P_i \times P_j} (g'(x) - g'(y))^2 W_G(x,y)\, 
        \mathrm{d}x\, \mathrm{d}y}{2 \sum_{i,j \in V} \int_{P_i \times P_j}
        g'(x)^2 W_G(x,y)\,\mathrm{d}x\, \mathrm{d}y}
        \\
        & = 
        \frac{\sum_{i,j \in V} (g(i) - g(j))^2 w_{ij}}{2 \sum_{i,j \in V} g(i)^2 w_{ij}}.
    \end{align*}  
    Hence, we get the inequality $\lambda_{W_G}^{\max} \geq \lambda_G^{\max}$.
    
    On the other hand, given a nonzero function $f \in L^{\infty}(I)$, define 
    the function $\map{F}{V}{\mathbb{R}}$ by $F(i) = \int_{P_i} f(x)\, \mathrm{d}x$,
    for every $i \in V$.
    Then the definition of $\lambda_G^{\max}$ gives the inequality
    \begin{equation*}
        \frac{1}{2} \sum_{i,j \in V} (F(i) - F(j))^2 w_{ij} 
        \leq \lambda_G^{\max} \sum_{i,j \in V} F(i)^2 w_{ij},
    \end{equation*}
    that is,
    \begin{equation*}
        \sum_{i,j \in V} F(i)^2 w_{ij} - \sum_{i,j \in V} F(i) F(j) w_{ij}
        \leq \lambda_G^{\max} \sum_{i,j \in V} F(i)^2 w_{ij},
    \end{equation*}
    and hence, we have
    \begin{equation} \label{ineq:largestEVforG}
        - \sum_{i,j \in V} F(i) F(j) w_{ij} 
        \leq (\lambda_G^{\max} - 1) \sum_{i,j \in V} F(i)^2 w_{ij}.
    \end{equation}
    Now note that
    \begin{align*}
        \sum_{i,j \in V} F(i) F(j) w_{ij} 
        & =
        \sum_{i,j \in V} \left(\int_{P_i} f(x)\, \mathrm{d}x\right) 
        \left(\int_{P_j} f(y)\, \mathrm{d}y\right) w_{ij}
        \\
        & =
        \sum_{i,j \in V} \int_{P_i \times P_j} f(x) f(y) W_G(x,y)\, \mathrm{d}x\, \mathrm{d}y
        \\
        & =
        \int_{0}^{1} \int_{0}^{1} f(x) f(y) W_G(x,y)\, \mathrm{d}x\, \mathrm{d}y,
    \end{align*}
    and that
    \begin{align*}
        \sum_{i,j \in V} F(i)^2 w_{ij}
        & =
        \sum_{i,j \in V} \left(\int_{P_i} f(x)\, \mathrm{d}x\right)^2 w_{ij}
        \\
        & \leq
        \frac{1}{n} \sum_{i,j \in V} \left(\int_{P_i} f(x)^2\, \mathrm{d}x\right) w_{ij}
        \tag{using the Cauchy--Schwarz inequality}
        \\
        & =
        \sum_{i,j \in V} \left(\int_{P_i} f(x)^2\, \mathrm{d}x\right) 
        \left(\int_{P_j} 1\, \mathrm{d}y\right) w_{ij}
        \\
        & =
        \sum_{i,j \in V} \int_{P_i \times P_j} f(x)^2 W_G(x,y)\, \mathrm{d}x\, \mathrm{d}y
        \\
        & =
        \int_{0}^{1} \int_{0}^{1} f(x)^2 W_G(x,y)\, \mathrm{d}x\, \mathrm{d}y.
    \end{align*}
    Thus, using the fact that the largest eigenvalue of the Laplacian of  a loopless graph
    is $\geq 1$ \cite{Chung-SpectralGraphTh97}*{Lemma 1.7(ii)}, 
    \eqref{ineq:largestEVforG} becomes
    \begin{equation*}
        - \int_{0}^{1} \int_{0}^{1} f(x) f(y) W_G(x,y)\, \mathrm{d}x\, \mathrm{d}y
        \leq (\lambda_G^{\max} - 1) 
        \int_{0}^{1} \int_{0}^{1} f(x)^2 W_G(x,y)\, \mathrm{d}x\, \mathrm{d}y,
    \end{equation*}
    which implies
    \begin{equation*}
        \frac{\norm[e]{df}^2}{\norm[v]{f}^2}
        = 1 - \frac{\int_{0}^{1} \int_{0}^{1} f(x) f(y) W_G(x,y)\, \mathrm{d}x\, \mathrm{d}y}
        {\int_{0}^{1} \int_{0}^{1} f(x)^2 W_G(x,y)\, \mathrm{d}x\, \mathrm{d}y}
        \leq 1 + (\lambda_G^{\max} - 1)  = \lambda_G^{\max}.
    \end{equation*}
    This proves that $\lambda_{W_G}^{\max} \leq \lambda_G^{\max}$, as desired.
\end{proof}

\begin{remark}
    Combining \cref{lemma:bipartiteGraphonChar} and \cref{lemma:compare-lambdaG&WG} 
    with the fact that a connected graph is bipartite if and only if the largest eigenvalue
    of its Laplacian is $2$, we conclude that a connected graph is bipartite
    if and only if its associated graphon is bipartite.
\end{remark}

\subsection{Bipartiteness ratio of graphs and the associated graphons}

Let $G = (V,w)$ be a connected weighted graph. Recall that $V = \{1, \dots, n\}$
with $n \geq 2$. We now obtain a characterization for the bipartite ratio $\beta_{W_G}$
of the associated graphon $W_G$ of $G$ in terms of certain elements of $I^n$, 
analogous to the notion of the fractional Cheeger constant introduced by 
Khetan and Mj \cite{Abhishek-Mahan24}.

For every $\alpha = (\alpha_1, \dots, \alpha_n), \gamma = (\gamma_1, \dots, \gamma_n)
\in I^n$ with $0 < \alpha + \gamma \leq 1$, where $0 = (0, \dots, 0), 
1 = (1, \dots, 1) \in I^n$, define
\begin{equation*}
    \tilde \beta_G(\alpha,\gamma) \coloneq 
    \frac{\sum_{i,j \in V} [2 \alpha_i \alpha_j + 2 \gamma_i \gamma_j 
    + (\alpha_i + \gamma_i)(1 - (\alpha_j + \gamma_j))] w_{ij}}
    {2 \sum_{i,j \in V} (\alpha_i + \gamma_i) w_{ij}}.
\end{equation*}
Note that $\sum_{i,j \in V} (\alpha_i + \gamma_i) w_{ij} 
= \sum_{i,j \in V} (\alpha_i + \gamma_i) \vol(i)$, which is positive, 
since $\vol(i) > 0$ for all $i$.

\begin{lemma} \label{lemma:beta-WG-tilde}
    Given a connected weighted graph  $G = (V,w)$, we have
    \begin{equation} \label{eq:beta-WG-tilde}
        \beta_{W_G} = \inf_{\substack{\alpha, \gamma \in I^n \\ 0 < \alpha + \gamma \leq 1}} 
        \tilde \beta_G(\alpha,\gamma).
    \end{equation}
\end{lemma}

\begin{proof}
    It suffices to prove that the sets
    \begin{equation*}
        A = \{\tilde \beta_G(\alpha,\gamma) : \alpha, \gamma \in I^n, 
        0 < \alpha + \gamma \leq 1\},
    \end{equation*}
    and 
    \begin{equation*}
        B = \{\beta_{W_G}(L,R) : L, R \text{ are measurable subsets of } I,  
        L \cap R = \emptyset, \mu_L(L \cup R) > 0\}
    \end{equation*}
    are equal. Let $\alpha = (\alpha_1, \dots, \alpha_n), 
    \gamma = (\gamma_1, \dots, \gamma_n)$ be elements of $I^n$ with 
    $0 < \alpha + \gamma \leq 1$. Then, observe that the sets 
    \begin{equation*}
        L = \bigcup_{i \in V} \left(\frac{i-1}{n}, \frac{i - 1 + \alpha_i}{n}\right) 
        \quad \text{and} \quad 
        R = \bigcup_{j \in V} \left(\frac{j - \gamma_j}{n}, \frac{j}{n}\right)
    \end{equation*} 
    are disjoint measurable subsets of $I$ and $\mu_L(L \cup R) > 0$, and thus, 
    $\beta_{W_G}(L,R)$ lies in the set $B$. We will show that $\beta_{W_G}(L,R)
    = \tilde \beta_G(\alpha, \gamma)$, so that we can conclude that 
    $\tilde \beta_G(\alpha, \gamma)$ also belongs to the set $B$. For that, note that
    \begin{align*}
        & \hspace{0.6cm}
        2 \eta(L \times L) + 2 \eta(R \times R) + \eta((L \cup R) \times (L \cup R)^c)
        \\
        & =
        \sum_{i,j \in V} [2 \mu_L(L \cap P_i) \mu_L(L \cap P_j) 
        + 2 \mu_L(R \cap P_i) \mu_L(R \cap P_j)
        \\
        & \quad + 
        \mu_L((L \cup R) \cap P_i) \mu_L((L \cup R)^c \cap P_j)] w_{ij}
        \\
        & =
        \sum_{i,j \in V} \left[2 \frac{\alpha_i}{n} \frac{\alpha_j}{n}
        + 2 \frac{\gamma_i}{n} \frac{\gamma_j}{n} + \left(\frac{\alpha_i + \gamma_i}{n}\right)
        \left(\frac{1 - (\alpha_j + \gamma_j)}{n}\right)\right] w_{ij}
        \\
        & =
        \frac{1}{n^2} \sum_{i,j \in V} [2 \alpha_i \alpha_j + 2 \gamma_i \gamma_j 
        + (\alpha_i + \gamma_i)(1 - (\alpha_j + \gamma_j))] w_{ij},
    \end{align*}
    and that 
    \begin{align*}
        2 \eta((L \cup R) \times I) 
        & = 
        2 \sum_{i,j \in V} \mu_L((L \cup R) \cap P_i) \mu_L(P_j) w_{ij}
        \\
        & =
        \frac{2}{n} \sum_{i,j \in V} \left(\frac{\alpha_i + \gamma_i}{n}\right) w_{ij}
        \\
        & =
        \frac{2}{n^2} \sum_{i,j \in V} (\alpha_i + \gamma_i) w_{ij}.
    \end{align*}
    Combining the above two equations gives us that 
    $\beta_{W_G}(L,R) = \tilde \beta_G(\alpha, \gamma)$, and this proves that 
    $A$ is a subset of $B$. To obtain the other inclusion, start with disjoint
    measurable subsets $L$ and $R$ of $I$ with $\mu_L(L \cup R) > 0$, and set
    $\alpha_i = n \mu_L(L \cap P_i)$ and $\beta_j = n \mu_L(R \cap P_j)$, for every 
    $i,j \in V$. Then the above calculations show that the elements 
    $\alpha = (\alpha_1, \dots, \alpha_n), \gamma = (\gamma_1, \dots, \gamma_n)$ 
    of $I^n$ are such that $0 < \alpha + \gamma \leq 1$ and $\beta_{W_G}(L,R)
    = \tilde \beta_G(\alpha, \gamma)$, implying that $B$ is a subset of $A$.
\end{proof}

From the arguments similar to that in the proof of \cref{lemma:beta-WG-tilde}, 
and the fact that $e_G(A,B) = n^2 \eta(\bigcup_{i \in A, j \in B} P_i \times P_j)$, 
for all subsets $A,B$ of $V$, it follows that
\begin{equation} \label{eq:betaG&betaG-tilde}
    \beta_G 
    = \min_{{\substack{\alpha, \gamma \in \{0,1\}^n \\ 0 < \alpha + \gamma \leq 1}}}
    \tilde \beta_G(\alpha,\gamma).
\end{equation}

The next lemma, which is similar to \cite{Abhishek-Mahan24}*{Lemma 4.1}, 
shows that the infimum in \cref{eq:beta-WG-tilde} is attained.

\begin{lemma} \label{lemma:betaWG-attained}
    Given any connected weighted graph $G = (V,w)$, there exist elements 
    $\alpha, \gamma$ of $I^n$ with $0 < \alpha + \gamma \leq 1$ such that
    $\beta_{W_G} = \tilde \beta_G(\alpha,\gamma)$.
\end{lemma}

\begin{proof}
    If $\beta_{W_G} = \frac{1}{2}$, then we have 
    $\beta_{W_G} = \tilde \beta_G(\alpha,\gamma)$ for 
    $\alpha = \gamma = \left(\frac{1}{2}, \dots, \frac{1}{2}\right) \in I^n$.

    Now assume that $\beta_{W_G} \neq \frac{1}{2}$, that is, $\beta_{W_G} < \frac{1}{2}$,
    by \cref{lemma:beta-leq-half}. Then, using \cref{lemma:beta-WG-tilde}, 
    given any positive integer $k$, there exist elements 
    $\alpha^{(k)} = (\alpha_1^{(k)}, \dots, \alpha_n^{(k)}), 
    \gamma^{(k)} = (\gamma_1^{(k)}, \dots, \gamma_n^{(k)})$ of $I^n$ with 
    $0 < \alpha^{(k)} + \gamma^{(k)} \leq 1$ such that
    \begin{equation} \label{ineq:seqI2n-betaWG}
        \beta_{W_G} \leq \tilde \beta_G(\alpha^{(k)},\gamma^{(k)}) 
        < \beta_{W_G} + \frac{1}{k}.
    \end{equation}
    Since $I^n$ is compact, the sequences $(\alpha^{(k)})$ and $(\gamma^{(k)})$ 
    have convergent subsequences in $I^n$, which we again denote by 
    $(\alpha^{(k)})$ and $(\gamma^{(k)})$, respectively, abusing the notation. 
    Suppose they converge to $\alpha$ and $\gamma$, respectively, in $I^n$. 
    Then for all $i,j \in V$, the sequences $(\alpha_i^{(k)})$ and $(\gamma_j^{(k)})$ 
    converge to $\alpha_i$ and $\gamma_j$, respectively, in $I$. 
    Consequently, if $(1 \geq)\, \alpha + \gamma > 0$, then the sequence
    $\left(\tilde \beta_G(\alpha^{(k)},\gamma^{(k)})\right)$ converges to 
    $\tilde \beta_G(\alpha,\gamma)$ in $\mathbb{R}$. Then, using \cref{ineq:seqI2n-betaWG},
    it follows that $\beta_{W_G} = \tilde \beta_G(\alpha,\gamma)$. 
    We now show that the case $\alpha + \gamma = 0$ is not possible. 
    
    If $\alpha + \gamma$ is $0$, then there is a positive integer $N$ such that 
    for all $k \geq N$ and $j \in V$, we have $\alpha_j^{(k)} + \gamma_j^{(k)} < \delta$, 
    where $\delta = \frac{1}{2} - \beta_{W_G}$. Hence, for all $k \geq N$, we get
    \begin{align*}
        & \hspace{0.6cm}
        \sum_{i,j \in V} [2 \alpha_i^{(k)} \alpha_j^{(k)} + 2 \gamma_i^{(k)} \gamma_j^{(k)} 
        + (\alpha_i^{(k)} + \gamma_i^{(k)})(1 - (\alpha_j^{(k)} + \gamma_j^{(k)}))] w_{ij}
        \\
        & \geq
        \sum_{i,j \in V} (\alpha_i^{(k)} 
        + \gamma_i^{(k)})(1 - (\alpha_j^{(k)} + \gamma_j^{(k)})) w_{ij}
        \\
        & \geq
        (1 - \delta) \sum_{i,j \in V} (\alpha_i^{(k)} + \gamma_i^{(k)}) w_{ij},
    \end{align*}
    which implies that 
    $\tilde \beta_G(\alpha^{(k)}, \gamma^{(k)}) \geq \frac{1 -\delta}{2}$, for all $k \geq N$.
    Combining this with \cref{ineq:seqI2n-betaWG} gives
    \begin{equation*}
        \frac{1 -\delta}{2} = \frac{1}{4} + \frac{\beta_{W_G}}{2} 
        < \beta_{W_G} + \frac{1}{k},
    \end{equation*}
    equivalently, $k < \frac{4}{1 - 2 \beta_{W_G}}$ for all $k \geq N$, which is impossible.
\end{proof}

We compare the bipartiteness ratios of graphs and the associated graphons 
in the following lemma, using certain ``suitable'' random variables. 
The analogous result for the Cheeger constants is discussed in \cite{Abhishek-Mahan24}*{Lemma 4.4}.

\begin{lemma} \label{lemma:compare-betaG&WG}
    For every loopless, connected weighted graph $G = (V,w)$, the following inequality holds.
    \begin{equation*}
        \frac{1}{4} \beta_{G} \leq \beta_{W_G} \leq \beta_G.
    \end{equation*}
\end{lemma}

\begin{proof}
    It is clear from \cref{lemma:beta-WG-tilde} and \cref{eq:betaG&betaG-tilde} 
    that $\beta_{W_G} \leq \beta_G$. We proceed to prove the other inequality.
    Let $\alpha = (\alpha_1, \dots, \alpha_n)$ and $\gamma = (\gamma_1, \dots, \gamma_n)$ 
    be elements of $I^n$ with $0 < \alpha + \gamma \leq 1$ satisfying 
    $\beta_{W_G} = \tilde \beta_G(\alpha, \gamma)$. 
    The existence of such elements is guaranteed by \cref{lemma:betaWG-attained}.

    Let $L_1, \dots, L_n$ and $R_1, \dots, R_n$ be independent random variables
    on some probability space $(\Omega, \mathcal{A}, P)$ such that 
    $P(L_i^{-1}(1)) = \alpha_i$, $P(L_i^{-1}(0)) = 1 - \alpha_i$, 
    $P(R_i^{-1}(1)) = \gamma_i$, and $P(R_i^{-1}(0)) = 1 - \gamma_i$, for all $1 \leq i \leq n$.
    Define random variables $X$ and $Y$ as follows.
    \begin{equation*}
        X = \sum_{i,j \in V} [2 L_i L_j + 2 R_i R_j + (L_i + R_i)(1 - L_j - R_j)] w_{ij},
    \end{equation*}   
    and
    \begin{equation*} 
        Y = 2 \sum_{i,j \in V} (L_i + R_i) w_{ij}.
    \end{equation*}
    Then, since the graph $G$ is loopless, the expectations of $X$ and $Y$ are
    \begin{equation*}
        E[X] = \sum_{i,j \in V} [2 \alpha_i \alpha_j + 2 \gamma_i \gamma_j 
        + (\alpha_i + \gamma_i)(1 - \alpha_j -\gamma_j)] w_{ij},
    \end{equation*}
    and
    \begin{equation*}
        E[Y] = 2 \sum_{i,j \in V} (\alpha_i + \gamma_i) w_{ij}. 
    \end{equation*}
    We will now show that the inequality $X(\omega) \geq \frac{1}{4} \beta_G Y(\omega)$
    holds for all $\omega \in \Omega$. Let $\omega \in \Omega$ be arbitrary.
    Consider the set $S = \{i \in V : L_i(\omega) = R_i(\omega) = 1\}$. 
    If $S$ is the empty set, then we get $X(\omega) \geq \beta_G Y(\omega)$
    from \cref{eq:betaG&betaG-tilde}, and we are done. Suppose that the set $S$ is nonempty.
    Denote by $x = (x_1, \dots, x_n)$ and $y = (y_1, \dots, y_n)$, the elements of $I^n$
    defined by 
    \begin{equation*}
        x_i = L_i(\omega) \quad \text{and} \quad y_i = 
        \begin{cases}
            0 & \text{if } i \in S,
            \\
            R_i(\omega) & \text{if } i \notin S,
        \end{cases} 
    \end{equation*}
    for all $1 \leq i \leq n$. Note that $x$ and $y$ are elements of $\{0,1\}^n$, 
    and they satisfy $0 < x + y \leq 1$. So, thanks to \cref{eq:betaG&betaG-tilde}, 
    it suffices to prove that $X(\omega) \geq \frac{1}{4} \tilde \beta_G(x,y) Y(\omega)$.
    Observe that
    \begin{align*}
        X(\omega) 
        & = 
        \sum_{i,j \notin S} 
        (2 L_i L_j + 2 R_i R_j + (L_i + R_i)(1 - L_j - R_j))(\omega) w_{ij}
        \\
        & \quad +
        \sum_{i \notin S, j \in S} 
        (2 L_i L_j + 2 R_i R_j + (L_i + R_i)(1 - L_j - R_j))(\omega) w_{ij}
        \\
        & \quad +
        \sum_{i \in S, j \notin S} 
        (2 L_i L_j + 2 R_i R_j + (L_i + R_i)(1 - L_j - R_j))(\omega) w_{ij}
        \\
        & \quad +
        \sum_{i \in S, j \in S} 
        (2 L_i L_j + 2 R_i R_j + (L_i + R_i)(1 - L_j - R_j))(\omega) w_{ij}
        \\
        & =
        \sum_{i,j \notin S} 
        (2 L_i L_j + 2 R_i R_j + (L_i + R_i)(1 - L_j - R_j))(\omega) w_{ij}
        \\
        & \quad +
        \sum_{i \notin S, j \in S} (L_i + R_i)(\omega) w_{ij}
        + \sum_{i \in S, j \notin S} 2 w_{ij} + \sum_{i \in S, j \in S} 2 w_{ij}
        \\
        & \geq
        \sum_{i,j \notin S} [2 x_i x_j + 2 y_i y_j + (x_i + y_i)(1 - x_j - y_j)] w_{ij}
        \\
        & \quad +
        \frac{1}{2} \sum_{i \notin S, j \in S} 2x_i w_{ij} 
        + \sum_{i \in S, j \notin S} (2x_j + 1 - x_j - y_j) w_{ij}
        + \sum_{i \in S, j \in S} 2 w_{ij}
        \\
        & \geq
        \frac{1}{2} \sum_{i,j \in V} 
        [2 x_i x_j + 2 y_i y_j + (x_i + y_i)(1 - x_j - y_j)] w_{ij},
    \end{align*}
    and that
    \begin{align*}
        Y(\omega) 
        & = 
        2 \sum_{i \notin S, j \in V} (L_i + R_i)(\omega) w_{ij}
        + 2 \sum_{i \in S, j \in V} (L_i + R_i)(\omega) w_{ij}
        \\
        & =
        2 \sum_{i \notin S, j \in V} (x_i + y_i) w_{ij} 
        + 2 \sum_{i \in S, j \in V} 2 (x_i + y_i) w_{ij}
        \\
        & \leq 
        2 \cdot 2 \sum_{i,j \in V} (x_i + y_i) w_{ij}.
    \end{align*}
    Thus, we arrive at the inequality 
    \begin{equation*}
        X(\omega) \geq \frac{1}{4} \tilde \beta_G(x,y) Y(\omega) 
        \geq \frac{1}{4} \beta_G Y(\omega). 
    \end{equation*}
    As this is true for all $\omega \in \Omega$, it implies that 
    $E[X] \geq \frac{1}{4} \beta_G E[Y]$, that is, 
    \begin{equation*}
        \frac{E[X]}{E[Y]} = \beta_{W_G} \geq \frac{1}{4} \beta_G,
    \end{equation*}
    using the fact that $E[Y]$ is positive.
\end{proof}

\begin{remark}
    Combining \cref{thm:main}, \cref{lemma:compare-lambdaG&WG} 
    and \cref{lemma:compare-betaG&WG} yields 
    \begin{equation*}
        \frac{\beta_G^2}{32} \leq 2 - \lambda_G^{\max} \leq 2 \beta_G,
    \end{equation*}
    which is the dual Cheeger--Buser inequality for graphs, up to a multiplicative constant.
\end{remark}

\section*{Acknowledgements}

The author is grateful to Jyoti Prakash Saha for suggesting the problem and 
having useful discussions. The author also acknowledges the fellowship 
from IISER Bhopal during the Integrated Ph.D. program.


\end{document}